\documentclass[11pt]{article}

\usepackage[centertags]{amsmath}
\usepackage{amsfonts}
\usepackage{amssymb}
\usepackage{amsthm}
\usepackage{newlfont}
\usepackage{easybmat}
\usepackage{verbatim}

\newtheorem{thm}{Theorem}[section]
\newtheorem*{thm*}{Theorem}
\newtheorem{cor}[thm]{Corollary}

\newtheorem{lem}[thm]{Lemma}

\newtheorem{prop}[thm]{Proposition}
\newtheorem*{prop*}{Proposition}

\newtheorem*{conj*}{Conjecture}

\newtheorem*{dfn*}{Definition}
\theoremstyle{definition}
\newtheorem{rem}[thm]{\textbf{Remark}}
\newtheorem*{rmk*}{Remark}
\newtheorem*{fact*}{Fact}

\theoremstyle{proof}

\newcommand{\vr}{\textrm{V.Rad.}}

\newcommand{\abs}[1]{\left\vert#1\right\vert}
\newcommand{\set}[1]{\left\{#1\right\}}
\newcommand{\brac}[1]{\left(#1\right)}
\newcommand{\scalar}[1]{\left \langle #1 \right \rangle}

\newcommand{\Real}{\mathbb{R}}
\def \RR {\mathbb R}
\newcommand{\eps}{\varepsilon}

\newcommand{\vol}{\textrm{Vol}}

\newcommand{\E}{\mathbb E}
\renewcommand{\P}{\mathbb P}

\newlength{\defbaselineskip}
\numberwithin{equation}{section}

\begin{document}

\title{Inner Regularization of Log-Concave Measures and Small-Ball Estimates}
\author{Bo'az Klartag\textsuperscript{1} and Emanuel Milman\textsuperscript{2}}

\date{}

\footnotetext[1]{School of Mathematical Sciences, Tel-Aviv
University, Tel Aviv 69978, Israel. Supported in part by the Israel
Science Foundation and by a Marie Curie Reintegration Grant from the
Commission of the European Communities. Email: klartagb@tau.ac.il}

\footnotetext[2]{Department of Mathematics, Technion - Israel
Institute of Technology, Haifa 32000, Israel. Supported by ISF, GIF, the Taub Foundation (Landau Fellow) and the E. and J. Bishop Research Fund. Email:
emilman@tx.technion.ac.il.}

\maketitle

\begin{abstract}
In the study of concentration properties of isotropic log-concave measures, it is often useful to first ensure that the measure has super-Gaussian marginals.
To this end, a standard preprocessing step is to convolve with a Gaussian measure, but this has the disadvantage of destroying small-ball information.
We propose an alternative preprocessing step for making the measure seem super-Gaussian, at least up to reasonably high moments, which does not suffer from this caveat: namely, convolving the measure with a random orthogonal image of itself. As an application of this ``inner-thickening", we recover Paouris' small-ball estimates.
\end{abstract}

\section{Introduction}

Fix a Euclidean norm $\abs{\cdot}$ on $\Real^n$, and let $X$ denote an isotropic random vector in $\Real^n$ with
log-concave density $g$. Recall that a random vector $X$ in
$\Real^n$ (and its density) is called isotropic if $\E X = 0$ and
$\E X \otimes X = Id$, i.e. its barycenter is at the origin and its
covariance matrix is equal to the identity one. Taking traces, we
observe that $\E|X|^2 = n$. Here and throughout we use $\E$ to denote expectation and $\P$ to denote probability.
A function $g: \Real^n \rightarrow \Real_+$ is called log-concave if $-\log g : \Real^n \rightarrow \Real \cup \set{+\infty}$ is convex. Throughout this work, $C$,$c$,$c_2$,$C'$, etc. denote universal positive numeric constants, independent of any other parameter and in particular the dimension $n$, whose value may change from one occurrence to the next.

Any high-dimensional probability distribution which is absolutely
continuous has at least one super-Gaussian marginal (e.g.
\cite{K_jems}). Still, in the study of concentration properties of
$X$ as above, it is many times advantageous to know that {\it all}
of the one-dimensional marginals of $X$ are super-Gaussian, at least
up to some level (see e.g.
\cite{PaourisSuperGaussian,GuedonEMilmanInterpolating,KlartagEMilmanLowerBoundsOnZp}).
By this we mean that for some $p_0 \geq 2$:
\begin{equation} \label{eq:goal0}
\forall 2 \leq p \leq p_0 \;\;\; \forall \theta \in S^{n-1} \;\;\;\;\; (E \abs{\scalar{X,\theta}}^p)^\frac{1}{p} \geq c (E \abs{G_1}^p)^\frac{1}{p} ~,
\end{equation}
where $G_1$ denotes a one-dimensional standard Gaussian random
variable and $S^{n-1}$ is the Euclidean unit sphere in $\RR^n$. It
is convenient to reformulate this using the language of
$L_p$-centroid bodies, which were introduced by E. Lutwak and G.
Zhang in \cite{LutwakZhang-IntroduceLqCentroidBodies} (under a
different normalization). Given a random vector $X$ with density $g$
on $\Real^n$ and $p \geq 1$, the $L_p$-centroid body $Z_p(X) =
Z_p(g) \subset \Real^n$ is the convex set defined via its support
functional $h_{Z_p(X)}$ by:
\[
h_{Z_p(X)}(y) = \brac{\int_{\Real^n} \abs{\scalar{x,y}}^p g(x) dx}^{1/p} ~,~ y \in \Real^n ~.
\]
More generally, the \emph{one-sided} $L_p$-centroid body, denoted
$Z^+_p(X)$, was defined in \cite{GuedonEMilmanInterpolating} (cf. \cite{HaberlLpIntersectionBodies}) by:
\[
 h_{Z^+_p(X)}(y) = \brac{2 \int_{\Real^n} \scalar{x,y}_+^p g(x) dx}^{1/p} ~,~ y \in \Real^n ~,
\]
where as usual $a_+ := \max(a,0)$. Note that when $g$ is even then both definitions above coincide, and that when the barycenter of $X$ is at the origin, $Z_2(X)$ is the Euclidean ball $B_2^n$ if and only $X$ is isotropic. Observing that the right-hand side of (\ref{eq:goal0}) is of the order of $\sqrt{p}$, we would like to have:
\begin{equation} \label{eq:goal1}
\forall 2 \leq p \leq p_0 \;\;\; Z^+_p(X) \supset c \sqrt{p} B_2^n
~,
\end{equation}
where $B_2^n = \{ x \in \RR^n ; |x| \leq 1 \}$ is the unit Euclidean
ball.

Unfortunately, we cannot in general expect to satisfy
(\ref{eq:goal1}) for $p_0$ which grows with the dimension $n$. This
is witnessed by $X$ which is uniformly distributed on the
$n$-dimensional cube $[-\sqrt{3}, \sqrt{3}]^n$ (the normalization
ensures that $X$ is isotropic), whose marginals in the directions of
the axes are uniform on a constant-sized interval. Consequently,
some preprocessing on $X$ is required, which on one hand transforms
it into another random variable $Y$ whose density $g$ satisfies
(\ref{eq:goal1}), and on the other enables deducing back the desired
concentration properties of $X$ from those of $Y$.

\medskip

A very common such construction is to convolve with a Gaussian, i.e.
define $Y := (X + G_n) / \sqrt{2}$, where $G_n$ denotes an
independent standard Gaussian random vector in $\Real^n$. In
\cite{KlartagCLP} (and in subsequent works like
\cite{KlartagCLPpolynomial,FleuryImprovedThinShell}), the Gaussian
played more of a regularizing role, but in
\cite{GuedonEMilmanInterpolating}, its purpose was to ``thicken from
inside" the distribution of $X$, ensuring that (\ref{eq:goal1}) is
satisfied for all $p \geq 2$ (see \cite[Lemma
2.3]{GuedonEMilmanInterpolating}). Regarding the transference of
concentration properties, it follows from the argument in the proof
of \cite[Proposition 4.1]{KlartagCLP} that:
\begin{equation} \label{eq:transfer-above}
\P(|X| \geq (1+t) \sqrt{n}) \leq C \P\brac{|Y| \geq \sqrt{\frac{(1+t)^2 + 1}{2}} \sqrt{n}} \;\;\; \forall t \geq 0 ~,
\end{equation}
and:
\begin{equation} \label{eq:transfer-below-bad}
\P(|X| \leq (1-t) \sqrt{n}) \leq C \P\brac{|Y| \leq \sqrt{\frac{(1-t)^2 + 1}{2}} \sqrt{n}} \;\;\; \forall t \in [0,1] ~,
\end{equation}
for some universal constant $C>1$. The estimate (\ref{eq:transfer-above}) is perfectly satisfactory for transferring (after an adjustment of constants) deviation estimates above the expectation from $|Y|$ to $|X|$. However, note that the right-hand side of (\ref{eq:transfer-below-bad}) is bounded below by $P(|Y| \leq \sqrt{n/2})$ (and in particular does not decay to 0 when $t \rightarrow 1$), and so (\ref{eq:transfer-below-bad}) is meaningless for transferring \emph{small-ball} estimates from $|Y|$ to $|X|$. Consequently, the strategies employed in \cite{KlartagCLP,KlartagCLPpolynomial,FleuryImprovedThinShell,GuedonEMilmanInterpolating} did not and could not deduce the concentration properties of $|X|$ in the small-ball regime. This seems an inherent problem of adding an independent Gaussian: small-ball information is lost due to the ``Gaussian-thickening".

\medskip

The purpose of this note is to introduce a different inner-thickening step, which does not have the above mentioned drawback. Before formulating it, recall that $X$ (or its density) is said to be ``$\psi_\alpha$ with constant $D>0$'' if:
\begin{equation} \label{eq:psi_alpha}
 Z_p(X) \subset D p^{1/\alpha} Z_2(X) \;\;\; \forall p \geq 2 ~.
\end{equation}
We will simply say that ``$X$ is $\psi_\alpha$'', if it is $\psi_\alpha$ with constant $D \leq C$, and not specify explicitly the dependence of the estimates on the parameter $D$. By a result of Berwald \cite{BerwaldMomentComparison} (or applying Borell's Lemma \cite{Borell-logconcave} as in \cite[Appendix III]{Milman-Schechtman-Book}), it is well known that any $X$ with log-concave density satisfies:
\begin{equation} \label{eq:Zp-inclusion}
1 \leq p \leq q \;\;\; \Rightarrow \;\;\; Z_p(X) \subset Z_q(X) \subset C \frac{q}{p} Z_p(X) ~.
\end{equation}
In particular, such an $X$ is always $\psi_1$ with some universal constant, and so we only gain additional information when $\alpha > 1$.

\begin{thm} \label{thm:main}
Let $X$ denote an isotropic random vector in $\Real^n$ with a
log-concave density, which is in addition $\psi_\alpha$ ($\alpha \in
[1,2]$), and let $X'$ denote an independent copy of $X$. Given $U
\in O(n)$, the group of orthogonal linear maps in $\Real^n$, denote:
\[
Y^U_\pm := \frac{X \pm U(X')}{\sqrt{2}} ~.
\]
Then:
\begin{enumerate}
\item
For any $U \in O(n)$, the concentration properties of $|Y^U_{\pm}|$ are transferred to $|X|$ as follows:
\[
P(|X| \geq (1+t) \sqrt{n}) \leq \brac{2 \max\brac{P(|Y^U_+| \geq (1+t) \sqrt{n}),P(|Y^U_{-}| \geq (1+t) \sqrt{n})}}^{1/2} \;\;\; \forall t \geq 0  ~,
\]
and:
\[
P(|X| \leq (1-t) \sqrt{n}) \leq \brac{2 \max\brac{P(|Y^U_+| \leq (1-t) \sqrt{n}),P(|Y^U_{-}| \leq (1-t) \sqrt{n})}}^{1/2} \;\;\; \forall t \in [0,1] ~.
\]
\item
For any $U \in O(n)$:
\begin{equation} \label{eq:Y-cond-outside}
Z_p^+(Y^U_{\pm}) \subset C p^{1/\alpha} B_2^n \;\;\; \forall p \geq 2 ~.
\end{equation}
\item
There exists a subset $A \subset O(n)$ with:
\[
\mu_{O(n)}(A) \geq 1 - \exp(-c n) ~,
\]
where $\mu_{O(n)}$ denotes the Haar measure on $O(n)$ normalized to have total mass $1$, so that if $U \in A$ then:
\begin{equation} \label{eq:Y-cond-inside}
Z_p^+(Y^U_\pm) \supset c_1 \sqrt{p} B_2^n \;\;\; \forall p \in [2,c_2 n^{\frac{\alpha}{2}}] ~.
\end{equation}
\end{enumerate}
\end{thm}

\begin{rem}
Note that when the density of $X$ is even, then $Y^U_+$ and $Y^U_-$
in Theorem \ref{thm:main} are identically distributed, which renders
the formulation of the conclusion more natural.
However, we do not know how to make the formulation simpler in the non-even case.
\end{rem}

\begin{rem}
Also note that $Y^U_{\pm}$ are isotropic random vectors, and that by the Pr\'ekopa--Leindler Theorem (e.g. \cite{GardnerSurveyInBAMS}), they have log-concave densities.
\end{rem}

\medskip

As our main application, we manage to extend the strategy in
 the second named author's previous work with O. Gu\'edon \cite{GuedonEMilmanInterpolating} to the small-ball
regime, and obtain:
\begin{cor} \label{cor:main}
Let $X$ denote an isotropic random vector in $\Real^n$ with log-concave density, which is in addition $\psi_\alpha$ ($\alpha \in [1,2]$). Then:
\begin{equation} \label{eq:all-deviation-X}
\P(\abs{|X| - \sqrt{n}} \geq t \sqrt{n}) \leq C \exp(-c n^{\frac{\alpha}{2}} \min(t^{2+\alpha},t)) \;\;\; \forall t \geq 0 ~,
\end{equation}
and:
\begin{equation} \label{eq:small-ball-X}
\P(|X| \leq \eps \sqrt{n}) \leq (C \eps)^{c n^{\frac{\alpha}{2}}} \;\;\; \forall \eps \in [0,1/C] ~.
\end{equation}
\end{cor}

Corollary \ref{cor:main} is an immediate consequence of Theorem \ref{thm:main} and the following result, which is the content of \cite[Theorem 4.1]{GuedonEMilmanInterpolating} (our formulation below is slightly more general, but this is what the proof gives):
\begin{thm*}[Gu\'edon--Milman]
Let $Y$ denote an isotropic random vector in $\Real^n$ with a
log-concave density, so that in addition:
\begin{equation} \label{eq:Y-cond}
c_1 \sqrt{p} B_2^n \subset Z_p^+(Y) \subset c_2 p^{1/\alpha} B_2^n \;\;\; \forall p \in [2,c_3 n^{\frac{\alpha}{2}}] ~,
\end{equation}
for some $\alpha \in [1,2]$. Then (\ref{eq:all-deviation-X}) and (\ref{eq:small-ball-X}) hold with $X = Y$ (and perhaps different constants $C,c>0$).
\end{thm*}

We thus obtain a preprocessing step which fuses perfectly with the approach in \cite{GuedonEMilmanInterpolating}, allowing us to treat all deviation regimes \emph{simultaneously} in a single unified framework.
We point out that Corollary \ref{cor:main} by itself is not
new. The \emph{large} positive-deviation estimate:
\[
P(|X| \geq (1+t) \sqrt{n}) \leq \exp(-c n^{\frac{\alpha}{2}} t) \;\;\; \forall t \geq C ~,
\]
was first obtained by G. Paouris in \cite{Paouris-IsotropicTail}; it
is known to be sharp, up to the value of the constants. The more
general deviation estimate (\ref{eq:all-deviation-X}) was obtained
in \cite{GuedonEMilmanInterpolating}, improving when $t \in [0,C]$
all previously known results due to the first named author and to
Fleury
\cite{KlartagCLP,KlartagCLPpolynomial,FleuryImprovedThinShell} (we
refer to \cite{GuedonEMilmanInterpolating} for a more detailed
account of these previous estimates). In that work, the convolution
with Gaussian preprocessing was used, and so it was not possible to
independently deduce the small-ball estimate
(\ref{eq:small-ball-X}). The latter estimate was first obtained by
Paouris in \cite{PaourisSmallBall}, using the reverse
Blaschke--Santal\'o inequality of J. Bourgain and V. Milman
\cite{Bourgain-Milman-vr-and-reverse-santalo}. In comparison, our
main tool in the proof of Theorem \ref{thm:main} is a covering
argument in the spirit of V. Milman's M-position
\cite{Milman-ReverseBM-CRAS,Milman-IsomorphicSymmetrization,Milman-GlobalQS-CRAS}
(see also \cite{Pisier-book}), together with a recent lower-bound on
the volume of $Z_p$ bodies obtained in our previous joint work
\cite{KlartagEMilmanLowerBoundsOnZp}.

\medskip
\noindent \textbf{Acknowledgement.} We thank Vitali
Milman and Olivier Gu\'edon for discussions.

\section{Key Proposition}

In this section, we prove the following key proposition:

\begin{prop} \label{prop:main}
Let $X,X'$ be as in Theorem \ref{thm:main}, let $U$ be uniformly distributed on $O(n)$, and set:
\[
Y := \frac{X + U(X')}{\sqrt{2}} ~.
\]
Then there exists a $c > 0$, so that:
\[
\forall C_1 > 0 \;\; \exists c_1 > 0 \;\; \forall p \in [2,c n^{\alpha/2}] \;\;\; \P(Z^+_p(Y) \supset c_1 \sqrt{p} B_2^n) \geq 1 - \exp(-C_1 n) ~.
\]
\end{prop}

Here, as elsewhere, ``uniformly distributed on $O(n)$" is with respect to the probability measure $\mu_{O(n)}$.

\medskip
We begin with the following estimate due to Gr\"{u}nbaum \cite{GrunbaumSymmetry}
(see also \cite[Formula (10)]{Fradelizi-Habilitation} or \cite[Lemma 3.3]{Bobkov-GaussianMarginals} for simplified proofs):
\begin{lem}[Gr\"{u}nbaum] \label{lem:barycenter}
Let $X_1$ denote a random variable on $\Real$ with log-concave density and barycenter at the origin. Then $\frac{1}{e} \leq \P(X_1 \geq 0) \leq 1-\frac{1}{e}$.
\end{lem}

Recall that the Minkowski sum $K + L$ of two compact sets $K,L
\subset \Real^n$ is defined as the compact set given by $\set{ x + y
; x \in K , y \in L}$. When $K, L$ are convex, the support
functional satisfies $h_{K+L} = h_K + h_L$.

\begin{lem} \label{lem:lem0}
With the same notations as in Proposition \ref{prop:main}:
\[
Z_p^+(Y) \supset \frac{1}{2 \sqrt{2} e^{1/p}} (Z_p^+(X) + U(Z_p^+(X))) ~.
\]
\end{lem}
\begin{proof}
Given $\theta \in S^{n-1}$, denote $Y_1 = \scalar{Y,\theta}$, $X_1 =
\scalar{X,\theta}$ and $X'_1 = \scalar{U(X'),\theta}$. By the
Pr\'ekopa--Leindler theorem (e.g. \cite{GardnerSurveyInBAMS}), all
these one-dimensional random variables have log-concave densities,
and since their barycenter is at the origin, we obtain by Lemma
\ref{lem:barycenter}:
\[
 h^p_{Z^+_p(Y)}(\theta) = 2 \E (Y_1)_+^p = \frac{2}{2^{p/2}} \E \brac{X_1 + X'_1}_+^p \geq \frac{2}{2^{p/2}} \E (X_1)_+^p \P(X'_1 \geq 0) \geq  \frac{2}{e 2^{p/2}} \E (X_1)_+^p ~.
\]
Exchanging the roles of $X_1$ and $X'_1$ above, we obtain:
\[
h^p_{Z^+_p(Y)}(\theta) \geq \frac{1}{e 2^{p/2}}
\max\brac{h^p_{Z^+_p(X)}(\theta),h^p_{Z^+_p(U(X'))}(\theta)} ~.
\]
Consequently:
\[
h_{Z^+_p(Y)}(\theta) \geq \frac{1}{\sqrt{2} e^{1/p}} \frac{h_{Z^+_p(X)}(\theta) + h_{Z^+_p(U(X'))}(\theta)}{2} ~,
\]
and since $Z^+_p(U(X')) = U(Z^+_p(X')) = U(Z^+_p(X))$, the assertion follows.
\end{proof}

Next, recall that given two compact subsets $K,L \subset \Real^n$, the covering number $N(K,L)$ is defined as the minimum number of translates of $L$ required to cover $K$. The volume-radius of a compact set $K \subset \Real^n$ is defined as:
\[
\vr(K) = \left( \frac{\vol(K)}{\vol(B_2^n)} \right)^{\frac{1}{n}} ~,
\]
measuring the radius of the Euclidean ball whose volume equals the volume of $K$. A convex compact set with non-empty interior is called a convex body, and given a convex body $K$ with the origin in its interior, its polar $K^\circ$ is the convex body given by:
\[
K^\circ := \set{ y \in \Real^n ; \scalar{x,y} \leq 1 \;\;\; \forall x \in K } ~.
\]
Finally, the mean-width of a convex body $K$, denoted $W(K)$, is
defined as $W(K) = 2 \int_{S^{n-1}} h_K(\theta) d\mu_{S^{n-1}}(\theta)$,
where $\mu_{S^{n-1}}$ denotes the Haar probability measure on $S^{n-1}$.
The following two lemmas are certainly well-known; we provide a
proof for completeness.

\begin{lem} \label{lem:lem1}
Let $K \subset \RR^n$ be a convex body with barycenter at the
origin, so that:
\[
N(K,B_2^n) \leq \exp(A_1 n) ~\text{ and }~ \vr(K) \geq a_1 > 0 ~.
\]
Then:
\[
N(K^\circ,B_2^n) \leq \exp(A_2 n) ~,
\]
where $A_2 \leq A_1 + \log (C / a_1)$, and $C>0$ is a universal constant.
\end{lem}
\begin{proof}
Set $K_s = K \cap -K$. By the covering estimate of H. K\"{o}nig and V. Milman \cite{Konig-Milman}, it follows that:
\[
N(K^\circ,B_2^n) \leq N(K_s^\circ,B_2^n) \leq C^n N(B_2^n,K_s) ~.
\]
Using standard volumetric covering estimates (e.g. \cite[Chapter
7]{Pisier-book}), we deduce:
\[
N(K^\circ,B_2^n) \leq C^n \brac{\frac{\vol(B_2^n +
K_s/2)}{\vol(K_s/2)}} \leq C^n N(K_s/2,B_2^n) \frac{\vol(2
B_2^n)}{\vol(K_s/2)} ~.
\]
By a result of V. Milman and A. Pajor \cite{MilmanPajor-NonSymmetric}, it is known that $\vol(K_s) \geq 2^{-n} \vol(K)$, and hence:
\[
N(K^\circ,B_2^n) \leq (8C)^n N(K,B_2^n) \vr(K)^{-n} \leq (8C /
a_1)^n \exp(A_1 n) ~,
\]
as required.
\end{proof}

\begin{lem} \label{lem:lem2}
Let $L$ denote any compact set in $\Real^n$ ($n \geq 2$), so
that $N(L,B_2^n) \leq \exp(A_1 n)$. If $U$ is uniformly distributed
on $O(n)$, then:
\[
P( L \cap U(L) \subset A_3 B_2^n) \geq 1 - \exp(-A_2 n) ~,
\]
where $A_2 = A_1 + (\log 2) / 2$ and $A_3 = C' \exp(6 A_1)$, for some universal constant $C'>0$.
\end{lem}
\begin{proof}[Proof Sketch]
Assume that $L \subset \cup_{i=1}^{\exp(A_1 n)} (x_i + B_2^n)$. Set $R = 4 C \exp(6 A_1)$, for some large enough constant $C>0$, and without loss of generality, assume that among all translates $\set{x_i}$, $\set{x_i}_{i=1}^N$ are precisely those points lying outside of $R B_2^n$. Observe that for each $i=1,\ldots,N$, the cone $\set{ t (x_i + B_2^n) ; t \geq 0}$ carves a spherical cap of Euclidean radius at most $1/R$ on $S^{n-1}$. By the invariance of the Haar measures on $S^{n-1}$ and $O(n)$ under the action of $O(n)$, it follows that for every $i,j \in \set{1,\ldots,N}$:
\[
P(U( x_i + B_2^n) \cap (x_j + B_2^n) \neq \emptyset) \leq \mu_{S^{n-1}}(B_{2/R}) ~,
\]
where $B_\eps$ denotes a spherical cap on $S^{n-1}$ of Euclidean radius $\eps$, and recall $\mu_{S^{n-1}}$ denotes the normalized Haar measure on $S^{n-1}$.  When $\eps < 1/(2C)$, it is easy to verify that:
\[
\mu_{S^{n-1}}(B_\eps) \leq (C \eps)^{n-1} ~,
\]
and so it follows by the union-bound that:
\[
P(L \cap U(L) \subset (R+1) B_2^n) \geq P( \forall i,j \in \set{1,\ldots,N} \;\;\; U( x_i + B_2^n) \cap (x_j + B_2^n) = \emptyset) \geq
1 - N^2 (2 C / R)^{n-1} ~.
\]
Since $N \leq \exp(2 A_1 (n-1))$, our choice of $R$ yields the desired assertion with $C' = 5 C$.
\end{proof}

It is also useful to state:
\begin{lem} For any density $g$ on $\Real^n$ and $p \geq 1$:
\begin{equation} \label{eq:trivial}
Z^+_p(g) \subset 2^{1/p} Z_p(g) \subset Z_p^+(g) - Z_p^+(g) ~.
\end{equation}
\end{lem}
\begin{proof}
The first inclusion is trivial. The second follows since $a^{1/p} + b^{1/p} \geq (a+b)^{1/p}$ for $a,b \geq 0$, and hence for all $\theta \in S^{n-1}$:
\[
h_{Z_p^+(g) - Z_p^+(g)}(\theta) = h_{Z_p^+(g)}(\theta) + h_{Z_p^+(g)}(-\theta) \geq 2^{1/p} h_{Z_p(g)}(\theta) ~.
\]
\end{proof}

The next two theorems play a crucial role in our argument. The first
is due to Paouris \cite{Paouris-IsotropicTail}, and the second to
the authors \cite{KlartagEMilmanLowerBoundsOnZp}:
\begin{thm*}[Paouris] \label{thm:Paouris}
With the same assumptions as in Theorem \ref{thm:main}:
\begin{equation} \label{eq:W-Paouris}
W(Z_p(X)) \leq C \sqrt{p} \;\;\; \forall p \in [2,c n^{\alpha/2}] ~.
\end{equation}
\end{thm*}

\begin{thm*}[Klartag--Milman] \label{thm:KM}
With the same assumptions as in Theorem \ref{thm:main}:
\begin{equation} \label{eq:Vol-KM}
\vr(Z_p(X)) \geq c \sqrt{p} \;\;\; \forall p \in [2, c n^{\alpha/2}] ~.
\end{equation}
\end{thm*}

\medskip

We are finally ready to provide a proof of Proposition \ref{prop:main}:

\begin{proof}[Proof of Proposition \ref{prop:main}]
Let $p \in [2,c n^{\alpha/2}]$, where $c>0$ is some small enough constant so that (\ref{eq:W-Paouris}) and (\ref{eq:Vol-KM}) hold. We will ensure that $c \leq 1$, so there is nothing to prove if $n=1$.
By (\ref{eq:trivial}), Sudakov's entropy estimate (e.g. \cite{Pisier-book}) and (\ref{eq:W-Paouris}), we have:
\begin{equation} \label{eq:proof1}
N(Z_p^+(X) / \sqrt{p}, B_2^n) \leq N(2^{1/p} Z_p(X) / \sqrt{p},
B_2^n) \leq \exp(\tilde{C} n W(2^{1/p} Z_p(X) / \sqrt{p})^2) \leq
\exp(C n) ~.
\end{equation}
Note that by (\ref{eq:trivial}) and the Rogers--Shephard inequality \cite{RogersShephard}, we have:
\[
2^{n/p} \vol(Z_p(X)) \leq \vol(Z_p^+(X) - Z_p^+(X)) \leq 4^n \vol(Z_p^+(X)) ~.
\]
Consequently, the volume bound in (\ref{eq:Vol-KM}) also applies to $Z_p^+(X)$:
\begin{equation} \label{eq:proof2}
\vr(Z_p^+(X)) \geq c_1 \sqrt{p} ~.
\end{equation}
By Lemma \ref{lem:lem1}, (\ref{eq:proof1}) and (\ref{eq:proof2}) imply that:
\[
N(\sqrt{p} (Z_p^+(X))^\circ, B_2^n) \leq \exp(C_2 n) ~.
\]
Consequently, Lemma \ref{lem:lem2} implies that if $U$ is uniformly distributed on $O(n)$, then for any $C_1 \geq C_2 + (\log 2)/2$, there exists a $C_3>0$, so that:
\[
\P\brac{ Z_p^+(X)^\circ \cap U(Z_p^+(X)^\circ) \subset
\frac{C_3}{\sqrt{p}} B_2^n} \geq 1 - \exp(-C_1 n) ~,
\]
or by duality (since $T(K)^\circ = (T^{-1})^*(K^\circ)$ for any linear map $T$ of full rank), that:
\begin{eqnarray*}
& & \P\brac{Z_p^+(X) + U(Z_p^+(X)) \supset C_3^{-1} \sqrt{p} B_2^n} \\
& \geq &  \P\brac{\textrm{conv}(Z_p^+(X) \cup U(Z_p^+(X))) \supset
C_3^{-1} \sqrt{p} B_2^n} \geq 1 - \exp(-C_1 n) ~.
\end{eqnarray*}
Lemma \ref{lem:lem0} now concludes the proof.
\end{proof}

\section{Remaining Details}

We now complete the remaining (standard) details in the proof of Theorem \ref{thm:main}.

\begin{proof}[Proof of Theorem \ref{thm:main}]
\hfill
\begin{enumerate}
\item
For any $U \in O(n)$ and $t \geq 0$, observe that:
\begin{eqnarray*}
& & 2 \max\brac{\P\brac{\abs{\frac{X + U(X')}{\sqrt{2}}} \leq t},\P\brac{\abs{\frac{X - U(X')}{\sqrt{2}}} \leq t} } \\
&\geq&  \P\brac{\abs{\frac{X + U(X')}{\sqrt{2}}} \leq t} + \P\brac{\abs{\frac{X - U(X')}{\sqrt{2}}} \leq t} \\
&=& \P\brac{\frac{\abs{X}^2 + \abs{X'}^2}{2} + \scalar{X,U(X')} \leq t^2} + \P\brac{\frac{\abs{X}^2 + \abs{X'}^2}{2} - \scalar{X,U(X')} \leq t^2} \\
&\geq& \P\brac{\abs{X} \leq t \text{ and } \abs{X'} \leq t \text{ and} \scalar{X,U(X')} \leq 0} \\
&  & + \P\brac{\abs{X} \leq t \text{ and } \abs{X'} \leq t \text{ and} \scalar{X,U(X')} > 0} \\
&=& \P\brac{\abs{X} \leq t \text{ and } \abs{X'} \leq t} = \P\brac{\abs{X} \leq t}^2 ~.
\end{eqnarray*}
Similarly:
\[
2 \max\brac{\P\brac{\abs{\frac{X + U(X')}{\sqrt{2}}} \geq t},\P\brac{\abs{\frac{X - U(X')}{\sqrt{2}}} \geq t} } \geq \P\brac{\abs{X} \geq t}^2 ~.
\]
This is precisely the content of the first assertion of Theorem \ref{thm:main}.
\item
Given $\theta \in S^{n-1}$, denote $Y_1 = P_{\theta} Y^U_+$, $X_1 = P_{\theta} X$ and $X_2 = P_{\theta} U(X')$, where $P_\theta$ denotes orthogonal projection onto the one-dimensional subspace spanned by $\theta$. We have:
\begin{eqnarray*}
& & h_{Z_p(Y^U_+)}(\theta) = \brac{\E |Y_1|^p}^{\frac{1}{p}} = \brac{\E \abs{\frac{X_1 + X_2}{\sqrt{2}}}^p}^{\frac{1}{p}} \\
&\leq & \frac{1}{\sqrt{2}} \brac{(\E|X_1|^p)^\frac{1}{p} + (\E|X_2|^p)^\frac{1}{p}} = \frac{1}{\sqrt{2}} \brac{h_{Z_p(X)}(\theta) + h_{Z_p(U(X)))}(\theta) } ~.
\end{eqnarray*}
Employing in addition (\ref{eq:trivial}), it follows that:
\[
Z_p^+(Y^U_+) \subset 2^{1/p} Z_p(Y^U_+) \subset \frac{2^{1/p}}{\sqrt{2}} \brac{Z_p(X) + U(Z_p(X))} ~,
\]
and the second assertion for $Y^U_+$ follows since $Z_p(X) \subset C p^{\frac{1}{\alpha}} B_2^n$ by assumption. Similarly for $Y^U_{-}$.
\item
Given a natural number $i$, set $p_i = 2^i$. Proposition \ref{prop:main} ensures the existence of a constant $c > 0$, so that for any $C_1 > 0$, there exists a constant $c_1>0$, so that for any $p_i \in [2,c n^\frac{\alpha}{2}]$, there exists a subset $A_i \subset O(n)$ with:
\[
\mu_{O(n)}(A_i) \geq 1 - \exp(-C_1 n) ~,
\]
so that:
\[
\forall U \in A_i \;\;\; Z_{p_i}(Y^U_+) \supset c_1 \sqrt{p_i} B_2^n ~.
\]
Denoting $A_0 := \cap \set{A_i \; ; \; p_i \in [2,c n^\frac{\alpha}{2}]}$, and setting $A = A_0 \cap -A_0$, where $-A_0 := \set{-U \in O(n); U \in A_0}$,
it follow by the union-bound that:
\[
\mu_{O(n)}(A) \geq 1 - 2 \log(C_2 + n) \exp(-C_1 n) ~.
\]
By choosing the constant $C_1 > 0$ large enough, we conclude that:
\[
\mu_{O(n)}(A) \geq 1 - \exp(-C_3 n) ~.
\]
By construction, the set $A$ has the property that:
\[
\forall U \in A \;\;\; \forall p_i \in [2 , c n^{\frac{\alpha}{2}}] \;\;\; Z_{p_i}(Y^U_{\pm}) \supset c_1 \sqrt{p_i} B_2^n ~.
\]
Using (\ref{eq:Zp-inclusion}), it follows that:
\[
\forall U \in A \;\;\; \forall p \in [2 , c n^{\frac{\alpha}{2}}] \;\;\; Z_{p}(Y^U_{\pm}) \supset \frac{c_1}{\sqrt{2}} \sqrt{p} B_2^n ~,
\]
thereby concluding the proof of the third assertion.
\end{enumerate}

\end{proof}

\def\cprime{$'$}

\end{document}